\documentclass[a4paper]{amsart}
\pdfoutput=1
\usepackage[utf8]{inputenc}  
\usepackage{currfile}
\usepackage{graphicx}
\usepackage{upref}
\usepackage{url}
\usepackage[pdftex, colorlinks=true]{hyperref}

\newcommand{\Dt}{{\Delta t}}
\newcommand{\abs}[1]{\left\vert#1\right\vert}
\newcommand{\R}{\mathbb R}
\newcommand{\N}{\mathbb N}
\newcommand{\Z}{\mathbb Z}
\newcommand{\norm}[1]{\left\Vert#1\right\Vert}
\newcommand{\seq}[1]{\left\{#1\right\}}
\newcommand{\Dp}{D_+}
\newcommand{\test}{\varphi}
\newcommand{\eps}{\varepsilon}
\DeclareMathOperator*{\sgn}{sign}
\newcommand{\sign}[1]{\sgn\left(#1\right)}
\newcommand{\dott}{\, \cdot\,}
\newcommand{\arxiv}[1]{\href{http://arxiv.org/pdf/#1}{arXiv:#1}} 

\newtheorem{theorem}{Theorem}[section]

\newtheorem{lemma}[theorem]{Lemma}

\numberwithin{equation}{section}     
\allowdisplaybreaks

\begin{document}
\title[FtL2LWR]{Follow-the-Leader models can be viewed as a numerical approximation to the 
Lighthill--Whitham--Richards model for traffic flow}
\author[Holden]{Helge Holden}
\address[Holden]{\newline
    Department of Mathematical Sciences,
    NTNU Norwegian University of Science and Technology,
    NO--7491 Trondheim, Norway}
\email[]{\href{helge.holden@ntnu.no}{helge.holden@ntnu.no}} 
\urladdr{\href{https://www.ntnu.edu/employees/holden}{https://www.ntnu.edu/employees/holden}}

\author[Risebro]{Nils Henrik Risebro}
\address[Risebro]{\newline
 Department of Mathematics,
University of Oslo,
  P.O.\ Box 1053, Blindern,
  NO--0316 Oslo, Norway }
\email[]{\href{nilshr@math.uio.no}{nilshr@math.uio.no}} 

\date{\today} 



\thanks{Research was supported by the grant {\it Waves and Nonlinear Phenomena (WaNP)} from the Research Council of Norway. The research was done while the authors were at Institut Mittag-Leffler, Stockholm.}

\begin{abstract} 
We show how to view the standard Follow-the-Leader (FtL) model as a numerical method to compute numerically the solution of the Lighthill--Whitham--Richards (LWR) model for traffic flow.
As a result we offer a simple proof that   FtL models converge to the LWR model for traffic flow
when traffic becomes dense. The proof is based on techniques used in the analysis of numerical schemes for  conservation laws, and 
the equivalence of weak entropy solutions of  conservation laws in the Lagrangian and Eulerian formulation. 
\end{abstract}

\maketitle
\section{Introduction} \label{sec:intro}
There are two paradigms in the mathematical modeling of traffic flow. One is based on an individual modeling of each vehicle with the dynamics governed by the distance between adjacent vehicles. The other is based on the assumption of dense traffic where the vehicles are represented by a density function, and individual vehicles cannot be identified. The dynamics is governed by a local velocity function depending solely on the density. The first model is denoted the Follow-the-Leader (FtL) model, and the second is called the Lighthill--Whitham--Richards (LWR) model  \cite{LW_II, richards} for traffic flow. Further refinements and extensions of these models are available. Intuitively, it is clear that the the FtL model should approach or approximate the LWR model in the case of heavy traffic, and that is what is proved here. This problem has been extensively studied, see \cite{Argall_etal,AwKlarMaterneRascle,ColomboRossi,CristianiSahu, 1605.05883,active,FrancescoRosini,GoatinRossi,1211.4619,HoldenRisebro_LWR,Rossi}.  Using numerical methods for scalar conservation laws we show that FtL models appear naturally as a numerical approximation of the LWR model. Thus we offer a short and direct proof that the FtL model converges to the LWR model, and our analysis is based on a careful study of the relationship between weak solutions in Lagrangian and Eulerian variables.

In the LWR model vehicles are described by a density $\rho=\rho(t,x)$ where $x$ is the position along the road, and $t$ as usual denotes time. Locally, one assumes that the velocity is given by a function $v$ that depends on the density only, that is, $v=v(\rho)$. If we consider unidirectional traffic on a homogenous road without exits or entries, conservation of vehicles requires that the  dynamics is governed by the scalar conservation law
   \begin{equation*}
  \rho_t+\big(\rho v(\rho)\big)_x=0,
  \end{equation*}
  which constitutes the LWR model. It is often denoted as ``traffic hydrodynamics'' due to its resemblance with fluid dynamics.
  
The FtL model can be described as follows. Consider $N$ vehicles with length $\ell$ and position
$z_1(t)<\dots<z_N(t)$ on the real axis with dynamics given by
\begin{align*}
  \dot z_i&=v\Bigl(\frac{\ell}{z_{i+1}-z_i}\Bigr) \ \text{ for $i=1,\dots, N-1$,} \\
  \dot{z}_N&=v_{\max}
\end{align*}
Here $v$ denotes a given velocity function with maximum $v_{\max}$, perhaps the speed
limit. Our proofs are considerably simpler when we have a
uniform bound on $z_{i+1}(t)-z_i(t)$. Having empty road ahead of the
first car would mean that ``$z_{N+1}-z_N=\infty$''. This is the same
as imposing $\dot{z}_N=v_{\max}$, and in this case $z_{i+1}(t)-z_i(t)$
would not be bounded by a constant independent of time. Therefore we 
will in this paper assume that we model one of two alternatives: \\
{\bf Periodic case:}  We are in the periodic case in which $z_i\in [a,b]$ for some
  interval $[a,b]$, and 
  \begin{equation*}
    \dot{z}_N(t) = v\left(\frac{\ell}{b-z_N(t)-a+z_1(t)}\right).
  \end{equation*}
{\bf Non-periodic case:} We imagine that there are infinitely many
  vehicles to the right of $z_N$, the distance between each of these vehicles
  is $M\ell$, for a finite, but arbitrary, constant $M>1$. In this case
  \begin{equation*}
    \dot{z}_N(t) = v\left(\frac{1}{M}\right).
  \end{equation*}

Introduce $y_i=(z_{i+1}-z_i)/\ell$ for $i=1,\dots, N-1$, to obtain
\begin{equation*}
  \dot y_i(t)=\frac1\ell\bigl(v(1/y_{i+1}(t))-v(1/y_i(t))\bigr).  
\end{equation*}
In this paper we analyze the limit of this system of ordinary differential equations when $N\to\infty$. There are two ways to proceed. 

We may analyze this system directly, in what we call the semi-discrete
case, see Section \ref{subsec:semidisc}.  
By using methods from the theory of numerical methods for scalar
conservation laws we show that  the sequence
$\seq{y_i(t)}_{i=1}^{N-1}$ converges, as $\ell\to 0$ and $N\to\infty$, to a function $y(t,x)$
that satisfies the equation 
\begin{equation}\label{eq:Lagrangey}
\begin{cases}
    y_t-V(y)_x=0 & t>0,\ x\in [0,1],\\
    y(0,x)=y_0(x) & x\in [0,1],\\
  \end{cases}
\end{equation}
where $V(y)=v(1/y)$, and
with boundary condition
\begin{equation*}
  \begin{cases}
    y(t,1)=y(t,0) \ &\text{ in the periodic case,}\\
    y(t,1)=M &\text{else.}
  \end{cases}
\end{equation*}
Note that $x$ is the Lagrangian mass coordinate, so that the integer part
of $x/\ell$ measures how many cars there are to the left of $x$.

Equation \eqref{eq:Lagrangey} is an example of a hyperbolic conservation law. It is well-known that solutions develop singularities, denoted shocks, in finite time independent of the smoothness of the initial data. Thus one needs to study weak solutions, and design so-called entropy conditions to identify the unique weak physical solution.  For a scalar conservation law $u_t+f(u)_x=0$ with initial data 
$u|_{t=0}=u_0$, the unique weak entropy solution $u=u(t,x)$, which is an integrable function of bounded variation, satisfies the Kru\v{z}kov entropy condition
\begin{equation}
\int\int_0^\infty \big( \abs{u-k}\phi_t+\sgn{(u-k)}(f(u)-f(k))\phi_x\big) \, dtdx+\int \abs{u_0-k}\phi|_{t=0}\, dx \ge 0 
\end{equation}
for all real constants $k\in\R$, and all non-negative test functions $\phi\in C^\infty_0(\R\times [0,\infty))$.  See \cite{HoldenRisebro}.

As an alternative approach, see Section \ref{subsec:eulerscheme}, we may discretize the time derivative by a small
positive $\Dt$ and write $z^n_j \approx z_j(n\Delta t)$,
$y_j^{n}\approx y_j(n\Delta t)$, we have that
\begin{equation*}
  z^{n+1}_j = z^n_j + \Dt V^n_j,\ \text{ and }\  
  y_j^{n+1}=y_j^{n}+\frac{\Dt}{\ell}\Bigl(
  V^n_{j+1}-V^n_j\Bigr), 
\end{equation*}
where $V^n_j=V(y^n_j)$.
The key observation is that this is an approximation of the hyperbolic conservation law $y_t-V(y)_x=0$ by a monotone scheme, and from the classical result of Crandall--Majda \cite{CranMaj:Monoton}, see also \cite[Thm. 3.9]{HoldenRisebro}, we know that this scheme converges, as $\ell\to0$, $N\to \infty$, and $\Dt\to 0$, to the entropy solution of equation \eqref{eq:Lagrangey}, namely $y_t-V(y)_x=0$. 
Thus in both cases we obtain convergence to the same hyperbolic conservation law in Lagrangian coordinates.

Next we have to transform the result of the two approaches, both in
Lagrangian coordinates, to Eulerian coordinates. For smooth solutions
this is nothing but a simple exercise in calculus, but for weak
entropy solutions this is a deep result due to Wagner
\cite{Wagner}. To be specific, we introduce the Eulerian space
coordinate $z=z(t,x)$, with $z_x=y$ and $z_t=V(y)$. 
A straightforward (but formal) calculation reveals that the Eulerian functions satisfy
\begin{equation*} \label{eq:E_L1} 
  y_t=-\frac1{\rho^2}\bigl(
  \rho_t+\rho_z v\bigr), \quad V(y)_x=\frac1{\rho} v'(\rho)\rho_z, 
\end{equation*}
and hence
\begin{equation*}
  \rho_t+\bigl(\rho v(\rho) \bigr)_z=0,
\end{equation*}
which is nothing but the LWR model. 
These formal transformations are not valid in general for weak entropy solutions. However, thanks to the fundamental result of Wagner \cite{Wagner}, weak entropy solutions in Lagrangian coordinates transform into weak entropy solutions in Eulerian variables.  
The approach here bears some resemblance to the approach in \cite{HoldenRisebro_LWR}, where the proof is obtained in a grid-less manner, 
and it does not depend on the use of  Crandall--Majda and Wagner. 

\section{The model}\label{sec:model}
Let us first introduce the FtL model.  Consider $N$ vehicles moving on
a one-dimensional road. Their position is given as a function of time
$t$ as $z_1(t), \dots, z_N(t)$. For the moment (we shall actually show
that this is so below) we assume that $z_1(t)<z_2(t)<\cdots <
z_N(t)$. We introduce the ``local inverse density'' by
\begin{equation*}
  y_i = \frac{1}{\ell}\big(z_{i+1}-z_i\big), \ \ i=1,\ldots,N-1,
\end{equation*}
where $\ell$ is the length of each vehicle. The velocity of the vehicle at
$z_i$ is
assumed to be a function of the distance to the vehicle in front, at
$z_{i+1}$. This means that 
\begin{equation} \label{eq:FLT}
  \dot z_i(t)=v\Bigl(\frac{\ell}{z_{i+1}(t)-z_i(t)}\Bigr), \quad i=1,\dots, N-1.  
\end{equation}
Regarding the first vehicle, located at $z_N$, we either assume that
there are infinitely many equally spaced vehicles in front of it, i.e., $y_N=M$, or that we are in the
periodic setting in an interval $[a,b]$, so that the distance from the
vehicle at $z_N$ to the vehicle at $z_1$ is $b-z_N + z_1 -a$, i.e.,
$y_N=(b-z_N+z_1-a)/\ell$. We have
\begin{equation}\label{eq:FLT_N}
  \dot{z}_N(t)=v\Bigl(\frac{1}{y_N(t)}\Bigr).
\end{equation}
Regarding the velocity function $v$, we assume it to be a 
decreasing Lipschitz continuous function such that
\begin{equation}
\text{$v(0)=v_{\max}$ and
$v(\rho)=0$ for $\rho\ge 1$.}\label{eq:Vass}
\end{equation}
The prototypical example is $v(\rho)=v_{\max}\max\seq{0,1-\rho}$. We
define the velocity in Lagrangian variables by $V(y)=v(1/y)$. 
Observe that $V$ is  globally bounded, Lipschitz continuous
and increasing for $y\ge 1$, with a bounded Lipschitz constant $L_v$. 

Rewriting \eqref{eq:FLT} in terms of $\seq{y_i}$ we get
\begin{equation}
  \label{eq:FLTy}
  \dot{y}_i = \frac{1}{\ell}
  \left(V({y_{i+1}})-V(y_{i})\right),
  \qquad i=1,\ldots,N-1,
\end{equation}
and
\begin{equation}
  \label{eq:FLTyN}
  y_N=
  \begin{cases}
    M, & \text{non-periodic case,}\\
    \frac{1}{\ell}(b-z_N + z_1 - a), & \text{periodic case.}
  \end{cases}
\end{equation}
Let us also define the \emph{Lagrangian grid} $\seq{x_{i-1/2}}_{i=1}^N$ by
$x_{i-1/2}=(i-1)\ell$. We shall also assume throughout that there is a
constant $1\le K < \infty$, $K$ independent of $N$ and $\ell$, such that
\begin{equation}
  \label{eq:initialass}
  1\le y_j(0)\le K,\qquad \sum_{j=1}^{N-1}\abs{y_{j+1}(0)-y_j(0)}\le K.
\end{equation}

\subsection{The semi-discrete case}\label{subsec:semidisc}
In this section we show that the solution of the system \eqref{eq:FLTy}  of ordinary differential equations
converges to an entropy solution of \eqref{eq:CLy} as
$\ell\to 0$, and that ``$1/y$'' converges to an entropy solution of
\eqref{eq:LWRz}.

Concretely, we define the piecewise constant function
\begin{equation}\label{eq:helerommet}
y_\ell(t,x)=y_j(t), \quad x\in (x_{j-1/2},x_{j+1/2}].
\end{equation} 
We shall also
use the notation
\begin{equation*}
  \Dp h_j = \frac{1}{\ell}\big(h_{j+1}-h_j\big)
\end{equation*}
for the forward difference. Let 
\begin{equation*}
  y^+ = \max\seq{y,0}\ \text{ and }\ y^-
  = - \min\seq{y,0},
\end{equation*}
and let $H$ denote the Heaviside function 
\begin{equation*}
  H(y)=
  \begin{cases}
    0 &y\le 0,\\ 1 &y>0.
  \end{cases}
\end{equation*}
\begin{lemma}
  Let $y_j(t)$ solve the system \eqref{eq:FLTy}. Then
  \begin{subequations}
    \begin{align}
      \label{eq:pluss}
      \frac{d}{dt} \left(y_j-k\right)^+ &\le
      \Dp\left[H(y_j-k)\left(V(y_j)-V(k)\right)\right], \\
      \frac{d}{dt} \left(y_j-k\right)^- &\le 
      \Dp\left[-H(k-y_j)\left(V(y_j)-V(k)\right)\right], \label{eq:minus}
    \end{align}
  \end{subequations}
  for any constant $k$.
\end{lemma}
\begin{proof}
 Throughout we use the notation $V_j=V(y_j)$.
  We have that
  \begin{align*}
    \frac{d}{dt} \left(y_j-k\right)^+ &=
    \frac{1}{\ell} H(y_j - k)\left(V_{j+1}-V_j\right)\\
    &=\frac{1}{\ell} \left[H(y_{j+1}-k)(V_{j+1}-V(k)) -
      H(y_j-k)(V_j-V(k))\right] \\
    &\quad - \frac{1}{\ell}\left(H(y_{j+1}-k)-H(y_j-k)\right)
    \left(V_{j+1}-V(k)\right)\\
    &=\Dp\left[H(y_j-k)(V_j-V(k))\right]\\
    &\quad - \frac{1}{\ell}\left(H(y_{j+1}-k)-H(y_j-k)\right)
    \left(V_{j+1}-V(k)\right).
  \end{align*}
  Now
  \begin{align*}
    (H(y_{j+1}-k)&-H(y_j-k))
    \left(V_{j+1}-V(k)\right)\\
    & =
    \begin{cases}
      0 & \text{$y_j$, $y_{j+1}\ge k$ or $y_j$, $y_{j+1}\le k$,}\\
      V_{j+1}-V(k)  & y_j<k<y_{j+1},\\
      V(k)-V_{j+1} & y_{j+1}<k<y_j,
    \end{cases}\\
    &\ge 0,
  \end{align*}
  since $y\mapsto V(y)$ is increasing. This proves \eqref{eq:pluss}; estimate 
  \eqref{eq:minus} is proved similarly.
\end{proof}
Now define $y_j(t)=y_1(t)$ for $j<1$ and $y_j(t)=y_{N-1}(t)$ for
$j>N-1$ in the non-periodic case. In the periodic case we define
$y_j(t)$ by periodic extension. To save space, we also use the
convention that in the non-periodic case, sums over $j$ range over all
$j\in \Z$, while in the periodic case, sums range over
$j=1,\ldots,N-1$.
\begin{lemma}
  \label{lem:bnd}
  If $1\le y_j(0)\le K$, then $1\le y_j(t)\le K$ for $t>0$.
\end{lemma}
\begin{proof}
  From \eqref{eq:pluss} and \eqref{eq:minus} we have
  \begin{equation*}
    \frac{d}{dt}\sum_j \left(y_j(t)-k\right)^{\pm} \le 0.
  \end{equation*}
  Thus if $y_j(0)\le K$ for all $j$, then $y_j(t)<k$ for any constant
  $k>K$. Similarly $y_j(t)>k$ for any constant $k<1$ if $y_j(0)\ge 1$
  for all $j$. 
\end{proof}
\begin{lemma}\label{lem:l1cont}
  If $\seq{\tilde{y}_j(t)}_{j=1}^{N-1}$ is another solution of
  \eqref{eq:FLTy} and  \eqref{eq:FLTyN} with initial data $\tilde{y}_j(0)$, then
  \begin{equation}
    \label{eq:l1cont}
    \sum_j \abs{y_j(T)-\tilde{y}_j(T)}\le \sum_j \abs{y_j(0)-\tilde{y}_j(0)},
  \end{equation}
  for $T>0$.
\end{lemma}
\begin{proof}
  Adding \eqref{eq:pluss} and \eqref{eq:minus}, and observing that 
  \begin{equation*}
    (y-k)^++(y-k)^-=\abs{y-k}\ \text{ and }\ 
    H(y-k)-H(k-y) = \sign{y-k},
  \end{equation*}
  we find that
  \begin{equation}\label{eq:discentropy}
    \frac{d}{dt}\abs{y_j-k} \le \Dp\left[\sign{y_j-k}\left(V(y_j)-V(k)\right)\right].
  \end{equation}
  Set
  $q(y,k)=\sign{y-k}(V(y)-V(k))$. Choosing $k=\tilde{y}_j(\tau)$ in
  the inequality for $y_j(t)$ and $k=y_j(t)$ in the inequality for
  $\tilde{y}_j(\tau)$, and adding the two inequalities, give
  \begin{equation*}
    \left(\frac{d}{dt}+\frac{d}{d\tau}\right)
    \abs{y_j(t)-\tilde{y}_j(\tau)} \le \left(D_{+,1}+D_{+,2}\right)
    q\left(y_j(t),\tilde{y}_j(\tau)\right), 
  \end{equation*}
  where $D_{+,1(2)}$ denotes the difference with respect to the first
  (second) argument. Summing over $j$, multiplying with a non-negative
  test
  function $\test(t,\tau)$, where $\test\in C^\infty_0((0,\infty)^2)$,
  and integrating by parts yield
  \begin{equation*}
    \int_0^\infty\int_0^\infty \left(\test_t+\test_\tau\right) \sum_j
    \abs{y_j(t)-\tilde{y}_j(\tau)}\,d\tau dt \ge 0,
  \end{equation*}
  since taking the sum makes the right-hand side ``telescope''. Now we
  can use Kru\v{z}kov's trick, see \cite[Sec.~2.4]{HoldenRisebro}, and choose 
  \begin{equation*}
    \test(t,\tau)=\psi\left(\frac{t+\tau}{2}\right)\omega_\eps(t-\tau),
  \end{equation*}
  where $\psi\in C^\infty_0((0,\infty))$, $\psi\ge 0$ and
  $\omega_\eps$ is a standard mollifier, to obtain, as $\eps\to0$, that 
  \begin{equation*}
    \int_0^\infty \psi'(t)\sum_j \abs{y_j(t)-\tilde{y}_j(t)}\,dt \ge 0.
  \end{equation*}
  Choose $\psi$ to be a smooth approximation to the
  characteristic function of the interval $(t_1,t_2)\subset (0,T)$,
  to get
  \begin{equation*}
    \sum_j \abs{y_j(t_2)-\tilde{y}_j(t_2)}\le \sum_j \abs{y_j(t_1)-\tilde{y}_j(t_1)}.
  \end{equation*}
  The lemma follows by letting $t_1\downarrow 0$ and $t_2\uparrow T$. For details, see \cite[Sec.~2.4]{HoldenRisebro}.
\end{proof}
\begin{lemma}
  \label{lem:compact} 
  Assume that $1\le y_j(0)\le K$ and that
  $\sum_j\abs{y_{j+1}(0)-y_j(0)}\le K$ for some constant $K$
  independent of $\ell$. Then there is a sequence $\seq{\ell_i}$,
  where $\ell_i \to 0$ as $i\to \infty$, and there exists a function $y\in C([0,T];L^1([0,1])$ such that $y_{\ell_i}$ converges to $y$
  in $C([0,T];L^1([0,1])$.
\end{lemma}
\begin{proof}
  Lemma~\ref{lem:bnd} shows that $\seq{y_\ell}_\ell$ is bounded
  independently of $\ell$;  choosing $\tilde{y}_j=y_{j+1}$ and using
  Lemma~\ref{lem:l1cont} yields the $BV$ bound on
  $\seq{y_\ell(t)}_\ell$ uniformly in $\ell$ and $t$. Choosing
  $\tilde{y}_j(t)=y_j(t-\sigma)$ in Lemma~\ref{lem:l1cont} for some $0<\sigma<t$ gives
  \begin{align*}
    \norm{y_\ell(t,\dott)-y_\ell(t-\sigma,\dott)}_{L^1}&=
    \ell\sum_j \abs{y_j(t)-y_j(t-\sigma)} \\
    &\le \ell \sum_j \abs{y_j(\sigma)-y_j(0)}\\
    &\le \sum_j \int_0^\sigma \abs{V(y_{j+1}(\xi))-V(y_j(\xi))}\,d\xi
    \\
    &\le \norm{V}_{\mathrm{Lip}} \sum_j \int_0^\sigma
    \abs{y_{j+1}(\xi)-y_{j}(\xi)}\,d\xi\\
    & \le \norm{V}_{\mathrm{Lip}} \sigma \sum_j \abs{y_{j+1}(0)-y_j(0)}\\
    &\le \norm{V}_{\mathrm{Lip}} \sigma K.
  \end{align*}
  Hence the map $t\mapsto y_\ell(t,\dott)$ is $L^1$ Lipschitz
  continuous, with a Lipschitz constant independent of $\ell$. Thus by
  \cite[Thm.~A.11]{HoldenRisebro}, the family
  $\seq{y_\ell}_{\ell>0}$ is compact in $C([0,\infty);L^1([0,1]))$.
\end{proof}
Furthermore we assume that as $N$ increases, the initial position of
the vehicles are such that there is a function $y_0(x)$ such that
\begin{equation}
  \lim_{\ell\to 0} y_\ell(0,\dott)=y_0(\dott), \label{eq:limitass}
\end{equation}
and that this convergence is in $L^1([0,1])$. We also assume that $\norm{y_0}_{L^\infty([0,1])}\le K$,
without loss of generality we can then also assume that
$\norm{y_\ell(0,\dott)}_{L^\infty([0,1])}\le K$.

It is now straightforward, starting from the discrete entropy
inequality \eqref{eq:discentropy}, to show that any limit of
$\seq{y_\ell}_{\ell>0}$ is the unique entropy solution to
\eqref{eq:CLy} by following a standard Lax--Wendroff argument, see \cite[Thm.~3.4]{HoldenRisebro}.   
Thus the whole sequence $\seq{y_\ell}$ converges, and
the unique entropy solution to \eqref{eq:Lagrangey} is the limit
\begin{equation*}
  y=\lim_{\ell\to 0} y_\ell.
\end{equation*}
Introduce the Eulerian spatial coordinate $z$, given by the equations
\begin{equation*}
  \frac{\partial z}{\partial x}=y,\qquad \frac{\partial z}{\partial t}=V(y),
\end{equation*}
and the variable $\rho=1/y$.
We can now proceed following the argument of Wagner~\cite{Wagner} to
obtain that $\rho$ is the unique weak entropy solution to the LWR model
\begin{equation}\label{eq:LWRz}
  \begin{cases}
    \rho_t+\bigl(\rho v(\rho) \bigr)_z=0, &t>0,\\
    \rho(0,z)=\rho_0(z).
  \end{cases}
\end{equation}

We can also study the convergence in Eulerian coordinates directly by defining a discrete version of the 
transformation from Lagrangian to Eulerian coordinates. 
To define the discrete version of $\rho$, we need the approximate
Eulerian coordinate; $z_\ell(t,x)$.  Define
\begin{equation*}
  z_\ell(t,x)=\frac{1}{\ell} \big(x_{j+1/2}-x\big)z_j(t)+
\frac{1}{\ell} \big(x-x_{j-1/2}\big)  z_{j+1}(t), \ \text{for $x\in [x_{j-1/2},x_{j+1/2}]$,}
\end{equation*}
where $\seq{z_j(t)}$ solves \eqref{eq:FLT}. Then
\begin{align*}
  \frac{\partial z_\ell}{\partial t}&=\frac{1}{\ell}\big(x_{j+1/2}-x\big) V_j+\frac{1}{\ell}\big(x-x_{j-1/2}\big)
  V_{j+1},\\
  \frac{\partial z_\ell}{\partial x}&= y_j,
\end{align*}
for $x\in (x_{j-1/2},x_{j+1/2})$. The sequence $\seq{z_\ell}_{\ell>0}$
is  uniformly Lipschitz continuous. Hence
by the Arzel{\`a}--Ascoli theorem, it converges uniformly to a Lipschitz continuous limit
$z(t,x)$ satisfying $z_t=V(y)$ and $z_x=y$ almost
everywhere. Furthermore the map $x\mapsto z_\ell(t,x)$ is invertible,
with inverse $x_\ell(t,z)$. In the periodic case we set 
\begin{subequations}
  \label{eq:zlrdef}
  \begin{equation}
    z_{l,\ell}(t)=a,\ \
    z_r(t)=b, \label{eq:zlrdefper}
\end{equation}
otherwise we define 
\begin{equation}
  z_{l,\ell}(t)=z_\ell(t,0),\ \ 
  z_r(t)=b+tV(M)=z(t,1)=z_\ell(t,1).\label{eq:zlrdefnonper}
\end{equation}

Observe that $z_{l,\ell}(t)=z_\ell(t,0)\to z(t,0)=z_l(t)$ as $\ell\to
0$.
\end{subequations}
Define 
\begin{equation}
  \label{eq:rhosemi}
  \rho_\ell(t,z)=\frac{1}{y_\ell(t,x_\ell(t,z))} \ \text{ for $z\in[z_{l,\ell}(t),z_r(t)]$.}
\end{equation}
In the periodic case, we define $\rho_\ell$ by periodic continuation,
while in the non-periodic case we define
\begin{equation*}
  \rho_\ell(t,z)=
  \begin{cases}
    0 &z<z_{l,\ell},\\
    1/M &z>z_r.
  \end{cases}
\end{equation*}
Next we claim that 
\begin{equation}\label{eq:rhoelllim}
  \rho_\ell(t,z)\to \rho(t,z)=\tilde\rho(t,x(t,z))
\end{equation}
in $L^1([z_l,z_r])$ as $\ell\to 0$. To see this, define
$\tilde\rho_\ell(t,x)=1/y_\ell(t,x)$, and compute
\begin{align*}
  \norm{\rho(t,\dott)-\rho_\ell(t,\dott)}_{L^1} &=
  \norm{\tilde\rho(t,x(t,\dott))-\tilde\rho_\ell(t,x_\ell(t,\dott))}_{L^1}\\
  &\le
  \underbrace{\norm{\tilde\rho(t,x(t,\dott))-\tilde\rho(t,x_\ell(t,\dott))}_{L^1}}_{A}\\
 &\quad +\underbrace{\norm{
      \tilde\rho(t,x_\ell(t,\dott))-\tilde\rho_\ell(t,x_\ell(t,\dott))}_{L^1}}_{B}.
\end{align*}
We have that
\begin{align*}
  A&=\int_{z_l}^{z_r}
  \abs{\tilde\rho(t,x(t,z))-\tilde\rho(t,x_\ell(t,z))}\,dz\\
  &\le \Bigl(\int_{\min\seq{z_l,z_{l,\ell}}}^{\max\seq{z_l,z_{l,\ell}}} + 
  \int_{\max\seq{z_l,z_{l,\ell}}}^{z_r}\Bigr) 
  \abs{\tilde\rho(t,x(t,z))-\tilde\rho(t,x_\ell(t,z))}\,dz.
\end{align*}
Since $\tilde\rho_\ell$ and $\tilde\rho$ are both bounded by $1$, and
$z_{l,\ell}\to z_l$ as $\ell\to 0$, the first of these integrals tend
to zero.
Since $x_\ell\to x$ uniformly, the integrand tends to zero almost
everywhere, and is bounded by $2$. Hence by the dominated convergence
theorem, the last integral tends to zero. The same argument applies to
$B$. Thus the claim \eqref{eq:rhoelllim} is justified.

Summing up, we have shown the following result.
\begin{theorem}
  \label{thm:semidiscrete}
   Assume that the function $v$ satisfies \eqref{eq:Vass}.
  Let $\seq{y_j}_{j=1}^{N-1}$ satisfy \eqref{eq:FLTy}, with either
  periodic boundary conditions; $y_N(t)=y_1(t)$, or $y_N(t)=M$ for
  some fixed constant $M>1$.  Assume that the initial positions
  of the vehicles $\seq{z_i(0)}_{i=1}^{N}$ are such the we can define a
  bounded function $y_0$ by \eqref{eq:limitass}, and that
   \eqref{eq:initialass} holds, namely that the initial data are bounded
  with finite total variation.  \\
 (i) The piecewise constant (in space) function $y_\ell(t,x)$ defined by \eqref{eq:helerommet} converges 
 in $C([0,T];L^1([0,1])$ as $\ell\to0$  to the unique weak entropy solution $y$ of \eqref{eq:Lagrangey}. 
 The function $\rho=1/y$ satisfies the LWR model \eqref{eq:LWRz} in Eulerian variables. \\
 (ii) The function
  $\rho_\ell$ defined by \eqref{eq:rhosemi} converges in $C([0,T];L^1([0,1])$ as $\ell\to0$  to the unique
 weak entropy solution $\rho$ of \eqref{eq:LWRz}. 
\end{theorem}

\subsection{Analysis of the Euler scheme for \eqref{eq:FLT}}
\label{subsec:eulerscheme}
The simplest numerical method to approximate solutions of
\eqref{eq:FLT} is the forward Euler scheme, viz.,
\begin{equation}
  \label{eq:FLT_euler}
  z_i((n+1)\Dt)=z_i(n\Dt) + \Dt
  v\Bigl(\frac{\ell}{z_{i+1}(n\Dt)-z_i(n\Dt)}\Bigr),
\end{equation}
where $\Dt$ is a (small) positive number.

If we write the Euler scheme \eqref{eq:FLT_euler} in the $y$ variable,
we get
\begin{equation}
  \label{eq:FLTy_euler}
  y^{n+1}_i = y^n_i + \lambda
  \left(V^{n}_{i+1}-V^n_i\right),\qquad i=1,\ldots,N-1,
\end{equation}
where $y^n_i=y_i(t^n)$, $t^n=n\Dt$, $\lambda=\Dt/\ell$  and $V^n_i=V(y^n_i)$. As a (right)
boundary condition we use
\begin{equation}\label{eq:boundarycl}
  V^n_N =
  \begin{cases}
    V(M) & \text{non-periodic,}\\
    V^n_1   & \text{periodic.}
  \end{cases}
\end{equation}
%
For $t\ge 0$ and $x\in [0,(N-1)\ell]$ define the function
\begin{equation*}
  y_\ell(t,x)=y^n_i \quad (t,x)\in [t^n,t^{n+1})\times (x_{i-1/2},x_{i+1/2}].
\end{equation*}
Observe that we can rewrite \eqref{eq:FLTy_euler} as
\begin{equation*}
  y^{n+1}_i = \left(1-\lambda\theta^n_{i+1/2}\right) y^n_i + \lambda\theta^n_{i+1/2} y^n_{i+1},
\end{equation*}
where 
\begin{equation*}
  \theta^n_{i+1/2}=-\frac{V^n_{i+1}-V^n_i}{y^n_{i+1}-y^n_i} \ge 0,
\end{equation*}
and since $V$ is Lipschitz continuous, $\theta^n_{i+1/2}\le L_v$. Hence if the CFL-condition 
\begin{equation}
  \label{eq:CFL}
  \lambda L_v \le 1,
\end{equation}
holds, then $y^{n+1}_i$ is a convex combination of $y^n_i$ and $y^n_{i+1}$. Thus
the scheme \eqref{eq:FLTy_euler} is monotone. In passing, we note that
a consequence is that if $1\le y^0_i\le K$ for all $i$, then $1\le y^n_i\le K$
for all $i$.  Regarding the position of vehicles, this means that if
$z_{i}(0)\le z_{i+1}(0)-\ell$, then $z_{i}(t^n)\le z_{i+1}(t^n)-\ell$.
So from a road safety perspective, the model is rather optimistic.

We are now interested in taking the limit as $\ell\to 0$. We do this
by increasing the number of vehicles such that $(N-1)\ell=1$; furthermore
we assume that \eqref{eq:limitass} holds.  Now
the conditions are such that fundamental results  of Crandall and
Majda \cite{CranMaj:Monoton}, see also \cite[Thm. 3.9]{HoldenRisebro},
can be applied. Thus we know that there is a function $y\colon\R_0^+\times
[0,1] \to \R$, with $y\in C(\R^+;L^1([0,1]))$, such that 
\begin{equation*}
  y_\ell(t,x) \to y(t,x),
\end{equation*}
with the limit being in $C(\R^+;L^1([0,1]))$, 
and that $y$ is the unique entropy solution to the Cauchy problem
\begin{equation}
  \label{eq:CLy}
  \begin{cases}
    y_t - V(y)_x = 0, \qquad &t>0, x\in [0,1],\\
    y(0,x)=y_0(x).
  \end{cases}
\end{equation}
If we do not have periodic conditions, this is supplemented with the
boundary condition 
\begin{equation*}
  y(t,1) = M, \quad t>0.
\end{equation*}
We remark that since the characteristic speeds of \eqref{eq:CLy} are strictly
negative, this boundary condition can be enforced strictly. 

Note that the convergence of $y_\ell$ and the bounds $1\le y_\ell \le
M$, imply the convergence of $\tilde\rho_\ell = 1/y_\ell$ to some function
$\tilde\rho$. We now proceed to show how $\tilde\rho$ is related to
the solution of the LWR model.

We also define the discrete ``Lagrange to Euler'' map $\tilde{z}_\ell$ as
follows. Let 
\begin{equation*}
  \tilde{z}^n_{i+1/2}=\tilde{z}^n_{i-1/2} + \ell y^n_i, \quad
  \text{i.e.,}\ \tilde{z}^n_{i+1/2}=z^n_{i+1}.
\end{equation*}
Since $z^n_i$ solves \eqref{eq:FLT_euler}, we also have that
\begin{equation*}
  \tilde{z}^{n+1}_{i+1/2}=\tilde{z}^n_{i+1/2} + \Dt v^n_{i+1}.
\end{equation*}
Define $\tilde{z}_\ell(t^n,x_{i+1/2})=\tilde{z}^n_{i+1/2}$, and by bilinear
interpolation between these points. For later use we employ the
notation for the value of $\tilde{z}_\ell$ at the edges of the
``\emph{Lagrangian grid}'',
\begin{align*}
  \tilde{z}_{i+1/2}(t)&=\frac{1}{\Dt}\left((t^{n+1}-t)\tilde{z}^n_{i+1/2} +
    (t-t^n)\tilde{z}^{n+1}_{i+1/2}\right),\ \text{for $t\in [t^n,t^{n+1}]$},\\
  \tilde{z}^{n}(x)&=\frac{1}{\ell}\left((x_{i+1/2}-x) \tilde{z}^{n}_{i-1/2}
  + (x-x_{i-1/2})\tilde{z}^n_{i+1/2}\right), \ \text{for $x\in [x_{i-1/2},x_{i+1/2}]$.} 
\end{align*}
Observe that $\tilde{z}_{i-1/2}(t)$ coincides with the approximate trajectory
of the vehicle starting at $z_i(0)$ calculated by the Euler method \eqref{eq:FLT_euler}.
Since $y_\ell$ is bounded, we can
invoke the Arzel{\`a}--Ascoli theorem to establish the convergence 
\begin{equation*}
  \lim_{\ell\to 0} \tilde{z}_\ell(t,x)=z(t,x),
\end{equation*}
with the limit being in $C([0,T]\times[0,1])$ and $z\in\ C([0,T]\times[0,1])$, 
and that 
\begin{equation*}
  \frac{\partial z}{\partial x}=y, \qquad \frac{\partial z}{\partial t}=V(y),
\end{equation*}
weakly. We have that the map $x \mapsto \tilde{z}_\ell(t,x)$ is invertible for each $t$, we
denote the inverse map by $x_\ell$, so that
$x_\ell(t,z_\ell(t,x))=x$.
Define $z_{l,\ell}$ and $z_r$ as in \eqref{eq:zlrdef} and $\rho_\ell$
as in \eqref{eq:rhosemi}. 

Note that if $z\in
(z_{i-1/2}(t),z_{i+1/2}(t)]$ and $t\in [t^n,t^{n+1})$, then
\begin{equation*}
  x_\ell(t,z)\in (x_{i-1/2},x_{i+1/2}], \quad
  \rho_\ell(t,z)=\tilde\rho^n_i:=\frac{1}{y^n_i}.
\end{equation*}
As before we have that
\begin{equation*}
  \rho_\ell(t,z)\to \rho(t,z)=\tilde\rho(t,x(t,z))
\end{equation*}
in $L^1([z_l,z_r])$ as $\ell\to 0$. 

By Wagner's result \cite{Wagner}, we have proved the following theorem. 
\begin{theorem}
  \label{thm:fullydiscrete}
  Assume that the function $v$ satisfies \eqref{eq:Vass}.
  Let $\ell>0$ and $N\in\N$, let $\seq{z_j}_{j=1}^N$ satisfy
  \eqref{eq:FLT_euler}, and assume that either we are in the periodic case
  $z_j\in [0,1]$, or that $z_N$ satisfies the boundary condition
  \eqref{eq:FLT_N}, with $y_N=M$. Assume that the initial positions
  of the vehicles $\seq{z_i(0)}_{i=1}^{N}$ are such the we can define a
  bounded function $y_0$ by \eqref{eq:limitass}, and that
  \eqref{eq:initialass} holds.

  Define the function $\rho_\ell(t,z)$ by \eqref{eq:rhosemi}.
  Let $N$ and $\ell$ satisfy $(N-1)\ell=1$ and assume that the
  CFL-condition \eqref{eq:CFL} holds.

  As $\ell\to 0$, $\rho_\ell$
  converges in $C([0,\infty);L^1)$ to the
  unique entropy solution $\rho$ of the conservation law \eqref{eq:LWRz}.
\end{theorem}
To illustrate the ideas in this paper we show how the method works in
a concrete example. We have a periodic road in the interval $z\in
[-1,1]$, and choose to position $N$ vehicles in this interval so that 
\begin{equation*}
  \rho_\ell(0,z)\approx \frac12\left(\cos(\pi z)+1\right).
\end{equation*}
In Figure~\ref{fig:1} we show the Lagrangian grid and the
corresponding mapping to Eulerian coordinates for $N=20$, and $t\in
[0,2]$. 
\begin{figure}[h!]
  \centering
  \begin{tabular}{rl}
    \includegraphics[width=0.45\linewidth]{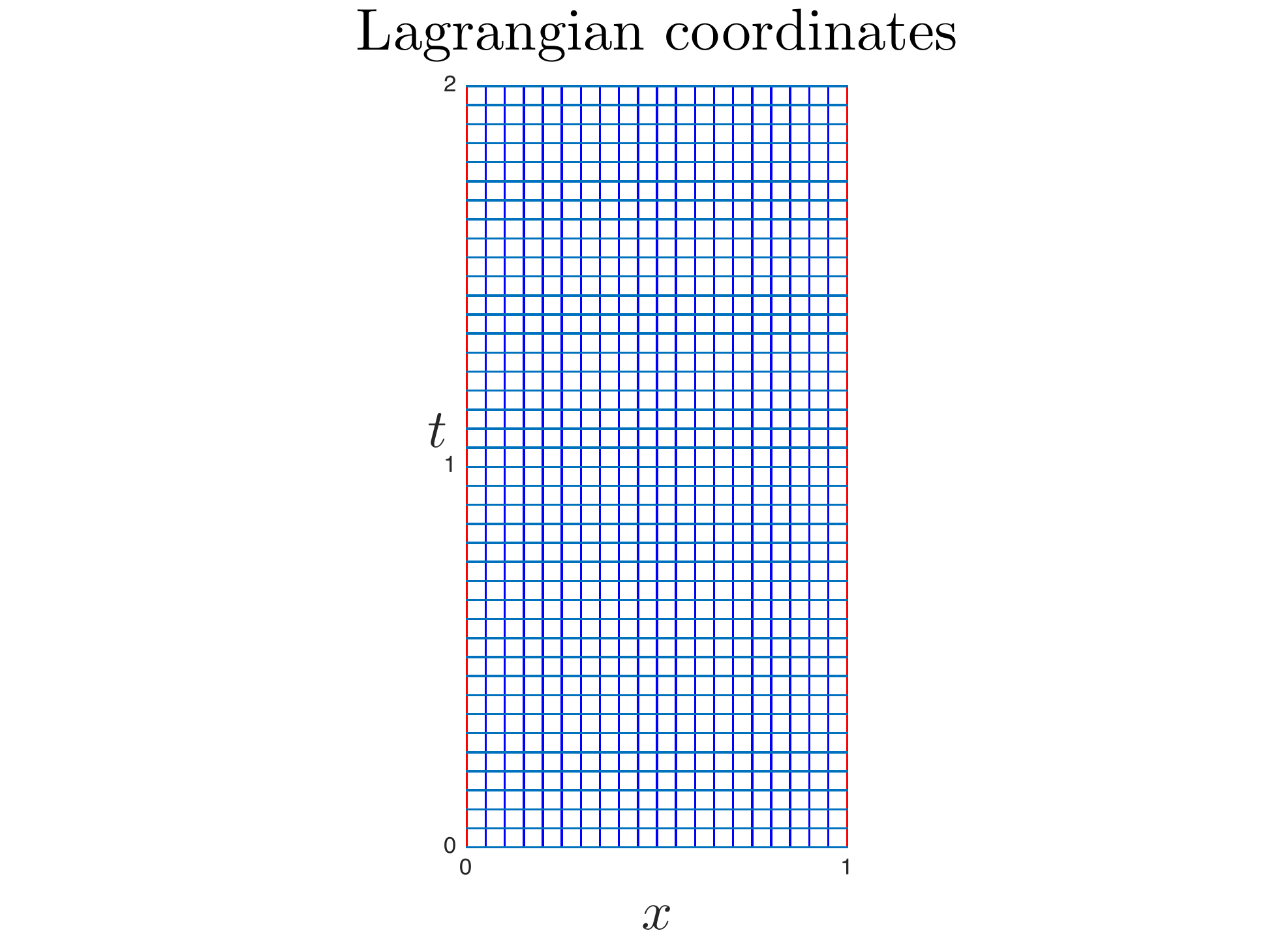}
    &    \includegraphics[width=0.45\linewidth]{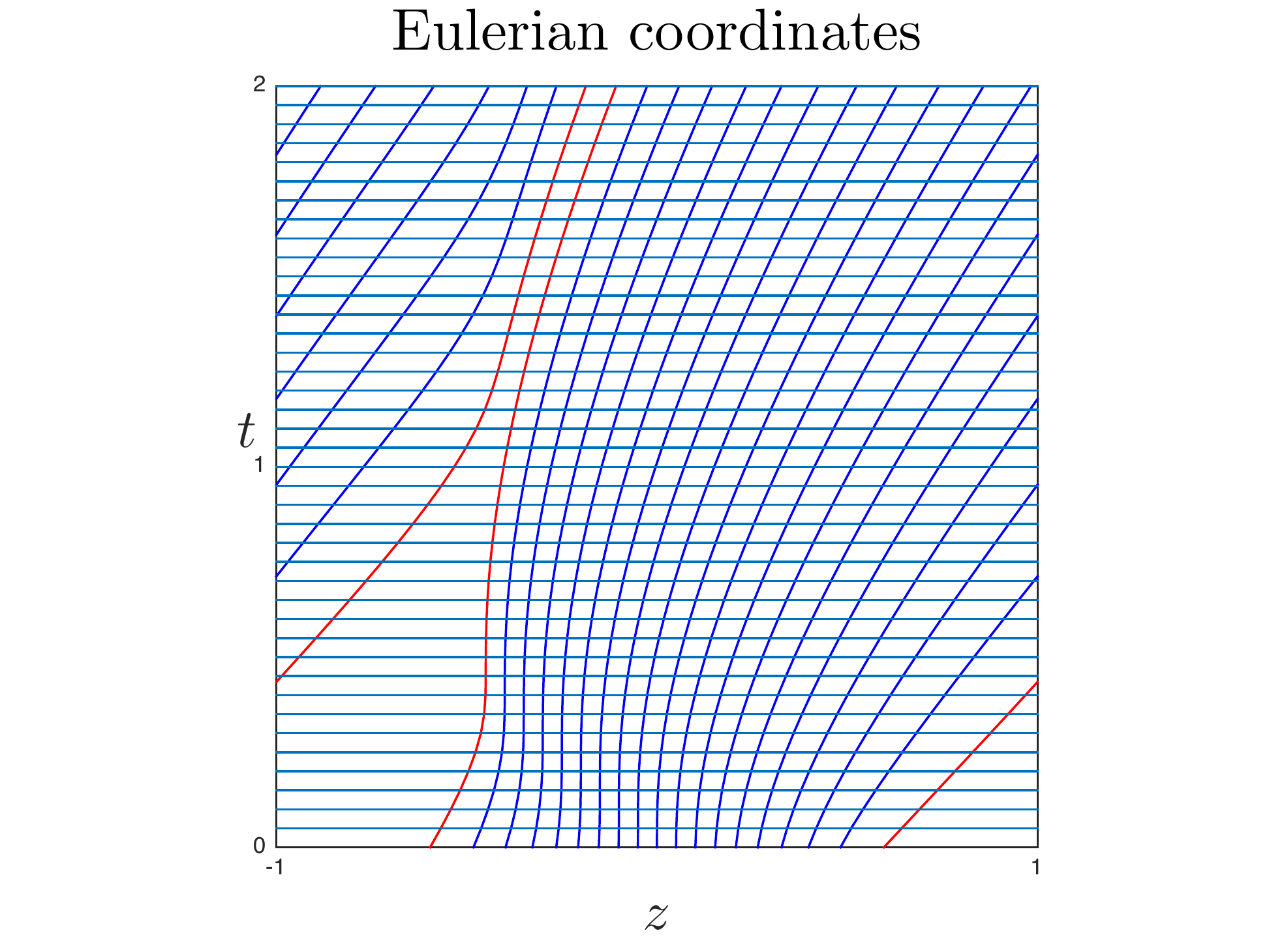}
  \end{tabular}
  \caption{Left: the Lagrangian grid $\seq{(t^n,x_{i-1/2})}_{i=1}^N$.  Right: the Eulerian grid $\seq{(t^n,z^n_{i-1/2})}_{i=1}^N$. In both cases $n=0,\ldots, 40$.}
  \label{fig:1}
\end{figure}
The vertical lines in the Eulerian coordinates are also the paths
followed by the vehicles, and the grid in Eulerian coordinates is the
result of applying the map $z_\ell$ to the rectangular grid depicted
in Lagrangian coordinates on the left. In Figure~\ref{fig:2}, we show
the approximate density $\rho_\ell$ at $t=0$ and $t=2$ in Eulerian
coordinates.
\begin{figure}[h!]
  \centering
  \includegraphics[width=0.5\linewidth]{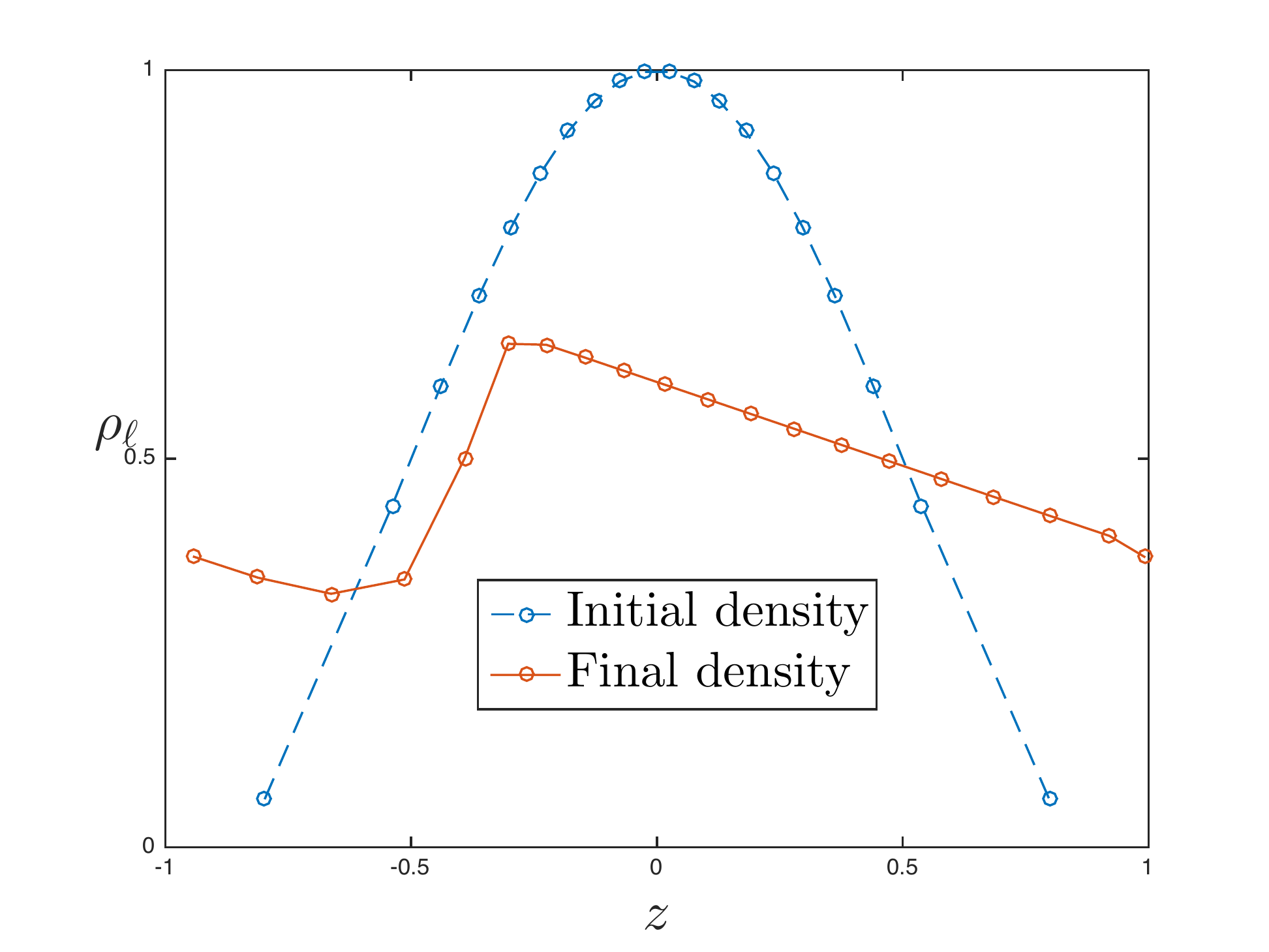}.
  \caption{The approximate density $\rho_\ell$ for $t=0$ and $t=2$ in Eulerian coordinates.}
  \label{fig:2}
\end{figure}
We see that the solution at $t=2$ approximates the ubiquitous ``$N$-wave''.


\end{document}